\documentclass[]{theclass}
\usepackage{graphicx, xfrac, lineno, float, subcaption, tasks, comment, xcolor, booktabs, multirow}
\usepackage[normalem]{ulem}
\usepackage{datetime}
\usepackage{colortbl}
\usepackage{enumerate}

\settasks{
	counter-format=(tsk[r]),
	label-width=4ex
}

\begin{document}

\begin{frontmatter}

\titledata{Extending perfect matchings to \\ Hamiltonian cycles in line graphs}{}           

\authordata{Mari\'{e}n Abreu}
{Dipartimento di Matematica, Informatica ed Economia\\ Universit\`{a} degli Studi della Basilicata, Italy}
{marien.abreu@unibas.it}
{$^{\dagger}$The research that led to the present paper was partially supported by a grant of the group GNSAGA of INdAM.}

\authordata{John Baptist Gauci}
{Department of Mathematics \\ University of Malta, Malta}
{john-baptist.gauci@um.edu.mt}
{}

\authordata{Domenico Labbate $^{\dagger}$}
{Dipartimento di Matematica, Informatica ed Economia\\ Universit\`{a} degli Studi della Basilicata, Italy}
{domenico.labbate@unibas.it}
{}

\authordata{Giuseppe Mazzuoccolo}
{Dipartimento di Informatica\\ Universit\`{a} degli Studi di Verona, Italy}
{giuseppe.mazzuoccolo@univr.it}
{}

\authordata{Jean Paul Zerafa}
{Dipartimento di Scienze Fisiche, Informatiche e Matematiche, \\ Universit\`{a} degli Studi di Modena e Reggio Emilia, Italy}
{jeanpaul.zerafa@unimore.it}
{}

\keywords{Line graph, Hamiltonian cycle, perfect matching.}
\msc{05C45, 05C70, 05C76.}

\begin{abstract}
A graph admitting a perfect matching has the Perfect-Matching-Hamiltonian property (for short the PMH-property) if each of its perfect matchings can be extended to a Hamiltonian cycle. In this paper we establish some sufficient conditions for a graph $G$ in order to guarantee that its line graph $L(G)$ has the PMH-property. In particular, we prove that this happens when $G$ is (i) a Hamiltonian graph with maximum degree at most $3$, (ii) a complete graph, or (iii) an arbitrarily traceable graph. Further related questions and open problems are proposed along the paper.

\end{abstract}

\end{frontmatter}

\section{Introduction}

The main property studied in this paper is related to two of the most studied concepts in graph theory: perfect matchings and Hamiltonian cycles. Let us recall that a  \textit{perfect matching} of a graph $G$ is a set of independent edges of $G$ that covers all the vertices in $G$, and a \textit{Hamiltonian cycle} is a cycle passing through all vertices of $G$. If such a cycle exists then $G$ is said to be \emph{Hamiltonian}.

The \textit{complete graph} on $n$ vertices, denoted by $K_n$, is the graph in which every two vertices are adjacent. For any graph $G$, $K_{G}$ denotes the complete graph on the same vertex set $V(G)$ of $G$. Let $G$ be of even order. A perfect matching of $K_{G}$ is said to be a \emph{pairing} of $G$. In \cite{ThomassenEtAl}, the authors say that a graph $G$ has the \emph{Pairing-Hamiltonian property} (for short the PH-property) if every pairing $M$ of $G$ can be extended to a Hamiltonian cycle $H$ of $K_{G}$ in which $E(H)-M\subseteq E(G)$, where $E(H)$ is the set of edges of $H$. Amongst other results, the authors show that the only cubic graphs having the PH-property are $K_{4}$, the complete bipartite graph $K_{3,3}$ and the 3-cube.
Adopting a similar terminology, we say that a graph $G$ admitting a perfect matching has the \emph{Perfect-Matching-Hamiltonian property} (for short the PMH-property) if every perfect matching of $G$ can be extended to a Hamiltonian cycle of $G$. We only consider graphs admitting a perfect matching to avoid trivial cases. This has already been studied in literature, and graphs having this property are also known as $F$-Hamiltonian, where $F$ is a perfect matching (see \cite{Haggkvist, Yang}).
Henceforth, if a graph has the Perfect-Matching-Hamiltonian property, we say that it is a PMH-graph or simply that it is PMH. Note that since every perfect matching of $G$ is a pairing of $G$, clearly, a graph having the PH-property is also a PMH-graph.

In the 1970s, Las Vergnas \cite{LasVergnas} (see Theorem \ref{Theorem LasVergnas}) and H\"{a}ggkvist \cite{Haggkvist} (see Theorem \ref{theorem haggkvist}) gave two sufficient Ore-type conditions for a graph to be PMH.

\begin{theorem}\cite{LasVergnas}\label{Theorem LasVergnas}
Let $G$ be a bipartite graph, with partite sets $U$ and $V$, such that $|U|=|V|=\frac{n}{2}\geq 2$. If for each pair of non-adjacent vertices $u\in U$ and $v\in V$ we have $deg(u)+deg(v)\geq \frac{n}{2}+2$, then $G$ is PMH.
\end{theorem}

\begin{theorem}\cite{Haggkvist}\label{theorem haggkvist}
Let $G$ be a graph, such that the order of $G$ is even and at least $4$. If for each pair of non-adjacent vertices $u$ and $v$ we have $deg(u)+deg(v)\geq n+1$, then $G$ is PMH.
\end{theorem}

Later on, in 1993, Ruskey and Savage \cite{RuskeySavage} asked whether every matching in the \linebreak$n$-dimensional hypercube $Q_{n}$, for $n\geq 2$, extends to a Hamiltonian cycle of $Q_{n}$. This was in fact shown to be true for $n=2,3,4$ (see \cite{Fink}) and for $n=5$ (see \cite{WangZhao}). Moreover, Fink \cite{Fink} also showed that $Q_{n}$ has the PH-property. This clearly implies that $Q_n$ is a PMH-graph, and thus answers a conjecture made by Kreweras (see \cite{kreweras}). Finally, Amar, Flandrin and Gancarzewicz in \cite{AFG} gave a degree sum condition for three independent vertices under which every matching of a graph lies in a Hamiltonian cycle. More results on PMH-graphs can be found in the paper by Yang \cite{Yang}.

The class of line graphs of connected graphs is a compelling class of graphs for which a great deal is known regarding Hamiltonicity and the existence of perfect matchings. Indeed, it is well-known that if $G$ is connected and has an even number of edges, then its line graph admits a perfect matching (see Section \ref{PMH-linegraphs} for more details), and so, in the sequel we shall tacitly assume that $G$ is connected and of even size. Furthermore, Hamiltonicity of a line graph $L(G)$ is another extensively studied property: a necessary and sufficient condition for Hamiltonicity in $L(G)$ is proved in \cite{HararyNashWilliams}, while Thomassen conjectured in \cite{Thomassen} that every $4$-connected line graph is Hamiltonian.

Along these lines, we here deal with the line graph of a graph $G$ and search for sufficient conditions on $G$ which result in $L(G)$ being PMH. We will prove that $L(G)$ is PMH in all of the following cases:
\begin{itemize}
\item $G$ is Hamiltonian with maximum degree $\Delta(G)$ at most $3$ (Theorem \ref{Theorem MaxDeg3}),
\item $G$ is a complete graph (Theorem \ref{Theorem CompleteGraphs}), and
\item $G$ is arbitrarily traceable from some vertex (Theorem \ref{randomeulerianPMH}).
\end{itemize}

In Section \ref{section kmm}, we shall also discuss the line graph of complete bipartite graphs. Further related results and open problems regarding graphs which are hypohamiltonian, Eulerian or with large maximum degree are discussed along the paper.
 
\subsection{Definitions and Notation}
All graphs considered in this paper are finite, simple (without loops or multiple edges) and connected. Most of our terminology is standard, and we refer the reader to \cite{BM} for further definitions and notation not explicitly stated.

Unless otherwise stated, we let the order of $G$ be $n$ and denote the set of vertices of $G$ by $\{v_1,v_2,\ldots,v_n\}$. For a graph $G$ and $N\subseteq E(G)$, $G-N$ represents the resulting graph after deleting the edges in $N$ from $G$.

A \textit{walk} (of length $k$) in a graph $G$ is a sequence $u_1,\ldots,u_{k+1}$ of vertices of $G$ with corresponding edge set $\{u_{i}u_{i+1}:i\in [k]\}$. If $u_1=u_{k+1}$, the walk is said to be \textit{closed} and is denoted by $(u_{1}, \ldots, u_{k+1}=u_{1})$. A \textit{path} on $t$ vertices, denoted by $P_{t}$, is a walk of length $t-1$ in which all the vertices and edges are distinct. We may also refer to $P_{t}$ as a $t$-\emph{path}. A \textit{cycle} of length $k$ is a closed walk of length $k$ in which all the vertices are distinct, except for the first and last. For simplicity, we denote a cycle of length $k$ by $(u_{1}, \ldots, u_{k})$, instead of $(u_{1}, \ldots, u_{k+1}=u_{1})$.

A \textit{tour} of $G$ is a closed walk having no repeated edges. A graph $G$ is \emph{Eulerian} if there is a tour that traverses all the edges of $G$, called an \emph{Euler tour}.
A \textit{dominating tour} of $G$ is a tour in which every edge of $G$ is incident with at least one vertex of the tour. In particular, a dominating tour which is 2-regular is referred to as a {\em dominating cycle}. In general, if a walk does not pass through some vertex $v$, we say that $v$ is \textit{untouched} or {\em uncovered}.

A \textit{clique} in a graph $G$ is a complete subgraph of $G$, and so $K_{n}$ may sometimes be referred to as an $n$-\emph{clique}.

\section{Line graphs of graphs with small maximum degree}\label{PMH-linegraphs}

The \textit{line graph} $L(G)$ of a graph $G$ is the graph whose vertices correspond to the edges of $G$, and two vertices of $L(G)$ are adjacent if the corresponding edges in $G$ are incident to a common vertex. For some edge $e\in E(G)$, we refer to the corresponding vertex in $L(G)$ as $e$, for simplicity, unless otherwise stated. A \textit{clique partition} of a graph $G$ is a collection of cliques of $G$ in which each edge of $G$ occurs exactly once. For any $v\in V(G)$, let $Q_{v}$ be the set of all the edges incident to $v$. Clearly, $Q_{v}$ induces a clique in $L(G)$ and $\mathcal{Q}=\{Q_{v}: v \in V(G) \text{ with degree at least }2\}$ is a clique partition of $L(G)$. We say that $\mathcal{Q}$ is the \textit{canonical clique partition} of $L(G)$. In the sequel, we shall refer to $Q_{v_{i}}$ simply as $Q_{i}$ and in order to avoid trivial cases, from now on we always assume that $G$ is a connected graph of order larger than $2$. In what follows, we shall also say that a clique $Q'\in\mathcal{Q}$ is intersected by a set of edges $N$ of $L(G)$, and by this we mean that $E(Q')\cap N\neq \emptyset$.

For a graph $F$, an \textit{$F$-decomposition} of $G$ is a collection of subgraphs of $G$ whose edges form a partition of $E(G)$ such that each subgraph in the collection is isomorphic to $F$. In general, it is not hard to show that every connected graph $G$ with $|E(G)|$ even has a $P_{3}$-decomposition. This is equivalent to saying that $L(G)$ has a perfect matching (see also Corollary 3 in \cite{Sumner}): indeed there is a natural bijection between the paths in a $P_{3}$-decomposition of $G$ and the edges of the corresponding perfect matching $M$ of $L(G)$, with the two edges in a $P_{3}$ corresponding to the two end-vertices of the respective edge in $M$. Since we are interested in line graphs which are PMH, a necessary condition is that $L(G)$ is Hamiltonian. Harary and Nash-Williams in \cite{HararyNashWilliams} showed that $L(G)$ is Hamiltonian if and only if $G$ admits a dominating tour.
In particular, this implies that if $G$ is Hamiltonian or Eulerian, then, $L(G)$ is also Hamiltonian, but the converse is not necessarily true (see also \cite{Chartrand, HararyNashWilliams, Sedlacek}).

The following technical lemma is the main tool we use to prove Theorem \ref{Theorem MaxDeg3} as well as a series of related results contained in this section. It describes a necessary and sufficient condition to extend a given perfect matching to a Hamiltonian cycle in subcubic graphs.

\begin{lemma}\label{Lemma TechnicalUncoveredCorrespondance}
Let $G$ be a connected graph such that $\Delta(G)\leq 3$. A perfect matching $M$ of $L(G)$ can be extended to a Hamiltonian cycle if and only if there exists a dominating cycle $D$ of $G$ such that the vertices in $G$ untouched by $D$ correspond to a subset of cliques in $\mathcal{Q}$ not intersected by $M$, where $\mathcal{Q}$ is the canonical clique partition of $L(G)$.
\end{lemma}

\begin{proof}
Let $M$ be a perfect matching of $L(G)$ which can be extended to a Hamiltonian cycle $H_{L}$ of $L(G)$. For some orientation of $H_{L}$, let $Q_{1},Q_{2}, \ldots, Q_{s}$ be the order in which $E(H_{L})$ intersects at least one edge of the cliques in $\mathcal{Q}$, where $s\in[n]$. Since $\Delta(G) \leq 3$, $\mathcal{Q}$ consists of 2-cliques and 3-cliques, implying that the sequence $Q_{1},Q_{2}, \ldots, Q_{s}$ does not have repetitions. We claim that $D=(v_1,v_2,\ldots,v_s)$ is a dominating cycle of $G$. Clearly, $D$ is a cycle, since consecutive cliques in the sequence $Q_{1},Q_{2}, \ldots, Q_{s}$ imply the existence of an edge between the corresponding two vertices in $D$. We then consider two cases. If every clique in $\mathcal{Q}$ is intersected by $E(H_{L})$, then $(v_1,v_2,\ldots,v_s)$ is a Hamiltonian cycle, since $s=n$. Therefore, consider the case when $\mathcal{Q}$ contains a clique, say $Q$, not  intersected by $E(H_{L})$. The edges of the other cliques in $\mathcal{Q}$ which are incident to a vertex in $Q$ must be intersected by $E(H_L)$, as otherwise the latter is not a Hamiltonian cycle of $L(G)$. Let these cliques be denoted by $Q_{j_{1}}, \ldots, Q_{j_{k}}$, for $k=2$ or $3$ and $j_{1},\ldots,j_{k}\in[s]$.
Let the corresponding vertices of $Q$ and $Q_{j_{1}}, \ldots, Q_{j_{k}}$, in $G$, be $v$ and $v_{j_{1}},\ldots,v_{j_{k}}$, respectively. Also, since $v\neq v_t$ for all $v_t$ in $D$, and $M$ is a perfect matching of $L(G)$, the vertices $v_{j_{1}},\ldots, v_{j_{k}}$ are in the cycle $D$ (not necessarily adjacent amongst themselves) and so the edges in $G$ having $v$ as an end-vertex have at least one end-vertex in $D$. Thus, since $v$ was arbitrary, $D$ is dominating. Moreover, every vertex in $G$ untouched by $D$ corresponds to a clique in $\mathcal{Q}$ not intersected by $E(H_{L})$, which is a subset of the cliques in $\mathcal{Q}$ not intersected by $M$.

Conversely, let $M$ be a perfect matching of $L(G)$ and let $D=(v_1,v_2,\ldots,v_s)$ be a dominating cycle in $G$, for some $s\leq n$, such that the untouched vertices correspond to a subset of the cliques in $\mathcal{Q}$ not intersected by $M$. Note that there exists a one-to-one mapping between the untouched vertices in $G$ and the unintersected cliques in $\mathcal{Q}$, which is not necessarily onto. We traverse the cliques in $\mathcal{Q}$ as follows. Let $Q$ be a clique in $\mathcal{Q}$, with corresponding vertex $v\in V(G)$. We consider three cases.

\textbf{Case 1:} $E(Q)\cap M\neq \emptyset$.\\
By our assumption, $v=v_i$ for some $i\in [s]$, and we traverse $Q$ ($=Q_{i}$) using the unique path joining $V(Q_{i-1})\cap V(Q_{i})$ and $V(Q_{i})\cap V(Q_{i+1})$ which contains $E(Q)\cap M$.

\textbf{Case 2:} $E(Q)\cap M=\emptyset$ and $v \in D$.\\
In this case, $v=v_j$ for some $j\in[s]$, and we traverse $Q$ ($=Q_j$) using the edge with end-vertices $V(Q_{j-1})\cap V(Q_{j})$ and $V(Q_{j})\cap V(Q_{j+1})$.

\textbf{Case 3:} $E(Q)\cap M=\emptyset$ and $v \not\in D$.\\
Since $M$ is a perfect matching, all the cliques in $\mathcal{Q}$ sharing a vertex with $Q$ (which must be triangles in this case) are intersected by $M$. These 3-cliques are traversed as in Case 1, and in this way the edges of $Q$ are not intersected.\\

We traverse all the cliques in $\mathcal{Q}$ in the above way and let the resulting sequence of edges be $H_{L}$. We claim that $H_{L}$ induces a Hamiltonian cycle of $L(G)$ containing $M$. By Case 1, $H_{L}$ contains $M$ and so every vertex of $L(G)$ is covered by $H_{L}$. Also, the sequence of cliques intersected by $E(H_{L})$, i.e. $Q_{1}, Q_{2}, \ldots, Q_{s}$, corresponds to the sequence of vertices in $D$, and so, since $D$ is connected and 2-regular, $H_{L}$ is a connected cycle, proving our claim.
\end{proof}
\begin{remark}\label{Remarkdominatingnottruefordeltagreaterthan3}
Note that Lemma \ref{Lemma TechnicalUncoveredCorrespondance} is not true in general for $\Delta(G) > 3$. An easy example is shown in Figure \ref{delta4}: indeed, an arbitrary perfect matching of $L(G)$ can be extended to a Hamiltonian cycle, i.e. $L(G)$ is PMH, but there is no dominating cycle in $G$.
\end{remark}

\begin{figure}[ht]
      \centering
      \includegraphics[width=0.485\textwidth]{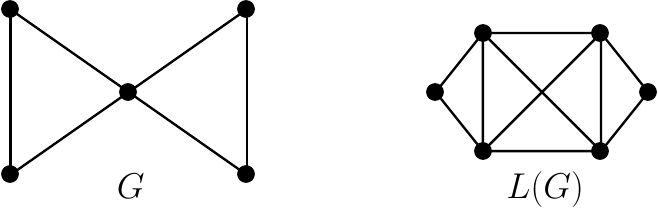}
	  \caption{A graph with maximum degree 4 whose line graph is PMH.}
      \label{delta4}
\end{figure}

By using Lemma \ref{Lemma TechnicalUncoveredCorrespondance}, we can furnish a first sufficient condition on $G$ assuring that its line graph is PMH.

\begin{theorem}\label{Theorem MaxDeg3}
Let $G$ be a Hamiltonian graph such that $\Delta(G)\leq 3$. Then, $L(G)$ is PMH.
\end{theorem}

\begin{proof}
Let $H$ be a Hamiltonian cycle of $G$. Given any perfect matching $M$ of $L(G)$, since the set of vertices untouched by $H$ in $G$ is empty, it is trivially a subset of the cliques in $\mathcal{Q}$ not intersected by $M$. Consequently, by Lemma \ref{Lemma TechnicalUncoveredCorrespondance}, $M$ can be extended to a Hamiltonian cycle of $L(G)$. Since $M$ was arbitrary, $G$ is PMH.
\end{proof}

In particular, Theorem \ref{Theorem MaxDeg3} applies for all Hamiltonian cubic graphs. However, in the cubic case we can say more. In 1964, Kotzig \cite{Kotzig} proved that the existence of a Hamiltonian cycle in a cubic graph is both a necessary and sufficient condition for a partition of $L(G)$ in two Hamiltonian cycles. We show the following.

\begin{corollary}\label{corollarymaxdeg3}
Let $G$ be a Hamiltonian cubic graph and $M$ a perfect matching of $L(G)$. Then, $L(G)$ can be partitioned in two Hamiltonian cycles, one of which contains $M$.
\end{corollary}

\begin{proof}
If we extend $M$ to a Hamiltonian cycle of $L(G)$ using the method described in Lemma \ref{Lemma TechnicalUncoveredCorrespondance}, we obtain a Hamiltonian cycle $H_{1}$ whose edge set intersects each triangle in $\mathcal{Q}$, since $G$ is Hamiltonian. Moreover, since $E(H_{1})$ intersects $Q\in\mathcal{Q}$ in one or two edges, the edges of $L(G)-E(H_{1})$ intersect $Q$ in two edges or one, respectively. Therefore, the edges in $L(G)-E(H_{1})$ induce a Hamiltonian cycle $H_{2}$ of $L(G)$ whose edges intersect the triangles in $\mathcal{Q}$ in the same order as the edges in $H_{1}$.
\end{proof}

When considering Theorem \ref{Theorem MaxDeg3}, one could wonder if the two conditions on the maximum degree and the Hamiltonicity of $G$ could be improved in some way.
First of all, we remark that our result is best possible in terms of the maximum degree of $G$: indeed, if $G$ is a Hamiltonian graph such that $\Delta(G)=4$, then, $L(G)$ is not necessarily PMH. For instance, consider the Hamiltonian graph in Figure \ref{Figure MaxDeg4NotPMH} having maximum degree $4$, and let $M$ be the perfect matching of $L(G)$ shown in the figure.

\begin{figure}[H]
      \centering
      \includegraphics[width=0.5\textwidth]{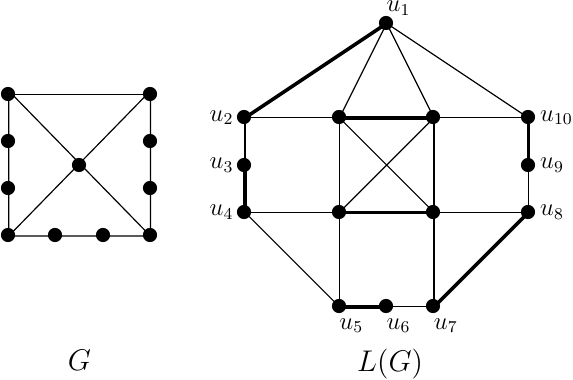}
	  \caption{A Hamiltonian graph with maximum degree 4 whose line graph is not PMH.}
      \label{Figure MaxDeg4NotPMH}
\end{figure}
Suppose $M$ can be extended to a Hamiltonian cycle. Then, it should include all edges incident to its vertices of degree $2$, and so it should contain the paths $u_{1},u_{2}, \ldots,u_{4}$ and $u_{5}, u_{6}, \ldots,u_{10}$. However, these two paths cannot be extended to a Hamiltonian cycle of $L(G)$ containing $M$, contradicting our assumption.

On the other hand, Hamiltonicity of $G$ in Theorem \ref{Theorem MaxDeg3} is not a necessary condition, since there exist non-Hamiltonian cubic graphs whose line graph is PMH. In particular, in Proposition \ref{prophypoham} we prove that hypohamiltonian cubic graphs are examples of such graphs. Let us recall that a graph $G$ is  \textit{hypohamiltonian} if $G$ is not Hamiltonian, but for every $v\in V(G)$, $G-v$ has a Hamiltonian cycle.

\begin{proposition}\label{prophypoham}
Let $G$ be a hypohamiltonian graph such that $\Delta(G)\leq 3$. Then, $L(G)$ is PMH.
\end{proposition}

\begin{proof}
Let $M$ be a perfect matching of $L(G)$. Since $|\mathcal{Q}|=|V(G)|$ is strictly larger than $|M|=\frac{|V(L(G))|}{2}\leq \frac{\frac{3}{2}|V(G)|}{2}$, there surely exists some clique $Q\in \mathcal{Q}$ which is not intersected by $M$. Let $v$ be the corresponding vertex in $G$. Since $G$ is hypohamiltonian, there exists a dominating cycle in $G$ which passes through all the vertices of $G$ except $v$, and so by Lemma \ref{Lemma TechnicalUncoveredCorrespondance},  $L(G)$ is PMH, since $M$ was arbitrary.
\end{proof}

Finally, another possible improvement of Theorem \ref{Theorem MaxDeg3} could be a weaker assumption on the length of the longest cycle of $G$ (i.e. the circumference of $G$, denoted by $circ(G)$). However, in Proposition \ref{Prop ExpandingAllButOne} we exhibit cubic graphs having circumference just one less than the order of $G$ whose line graphs are not PMH.

We will make use of the following standard operations on cubic graphs known as \linebreak$Y$-\textit{reduction} (shrinking a triangle to a vertex) and of its inverse, $Y$\textit{-extension} (expanding a vertex to a triangle), illustrated in Figure \ref{Figure Yoperations}.

\begin{figure}[ht]
      \centering
      \includegraphics[width=0.55\textwidth]{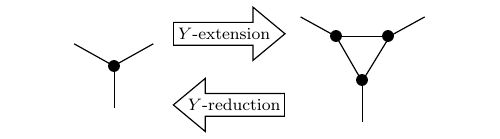}
	  \caption{$Y$-operations}
      \label{Figure Yoperations}
\end{figure}

For the proof of Proposition \ref{Prop ExpandingAllButOne}, we also need to show that each edge of $L(G)$, where $G$ is cubic and Hamiltonian, belongs to a perfect matching. This kind of property is extensively  studied in many papers and a graph $G$ is said to be \emph{1-extendable} if every edge in $G$ belongs to a perfect matching of $G$.
Theorem 2.1 in \cite{Plummer} states that every claw-free $3$-connected graph is $1$-extendable. By recalling that every line graph is a claw-free graph, we have, in particular, that $L(G)$ is $1$-extendable if $G$ is cubic and $3$-edge-connected. The generalisation to an arbitrary Hamiltonian cubic graph $G$ is not hard to achieve by using such a result, but here we prefer to present a direct short proof which is valid for any bridgeless cubic graph and which makes use of the following tool from the proof of Proposition 2 in \cite{Mazzuoccolo}.

\begin{remark}\cite{Mazzuoccolo}\label{Remark Graph Isomorphic to Original}
Let $G_{1}$ be a cubic graph of even size and $M$ a perfect matching of $L(G_1)$, with canonical clique partition $\mathcal{Q}$. The graph $G_2$ obtained by removing all the edges in $M$ from $L(G_{1})$ and then applying $Y$-reductions to all the triangles in $\mathcal{Q}$ not intersected by $M$, is isomorphic to $G_{1}$. 
\end{remark}

Remark \ref{Remark Graph Isomorphic to Original} follows by considering the natural bijection $\phi$ between $V(G_{1})$ and $\mathcal{Q}$, and the function $\psi_{M}$ between $\mathcal{Q}$ and $V(G_{2})$, where $\psi_{M}(Q)$, for $Q\in\mathcal{Q}$, is defined as follows. If $E(Q)\cap M=\emptyset$, $Q$ is mapped to the vertex in $G_{2}$ obtained after applying a $Y$-reduction to $Q$. Otherwise, if $E(Q)\cap M\neq\emptyset$, $Q$ is mapped to the vertex in $G_{2}$ corresponding to the vertex in $Q$ unmatched by $E(Q)\cap M$. It is not hard to prove that $\psi_{M}\circ\phi$ is an isomorphism between $G_{1}$ and $G_{2}$.

\begin{lemma}\label{Lemma TechnicalPMIntersectsTriangle}
Let $G$ be a bridgeless cubic graph of even size. Then, every edge of $L(G)$ belongs to a perfect matching.
\end{lemma}
\begin{proof}
Let $e\in E(L(G))$ and let $M$ be a perfect matching of $L(G)$. Assume $e \notin M$, otherwise the statement holds. The graph $L(G)-M$ is cubic and by Remark \ref{Remark Graph Isomorphic to Original} can be obtained by applying suitable $Y$-extensions to $G$. Since $G$ is bridgeless, and the resulting graph after applying $Y$-extensions to a bridgeless graph is again bridgeless, we have that $L(G)-M$ is bridgeless as well. Moreover, in \cite{Schonberger}, Sch\"{o}nberger proved that every bridgeless cubic graph is $1$-extendable: hence, there exists a perfect matching of $L(G)-M$ which contains $e$. Such a perfect matching is trivially also a perfect matching of $L(G)$ containing $e$.
\end{proof}

The following proposition shows that the Hamiltonicity condition in Theorem \ref{Theorem MaxDeg3} cannot be relaxed to any other condition regarding the length of the longest cycle in $G$. Indeed, starting from an appropriate cubic graph and performing suitable $Y$-extensions, we obtain a graph of circumference one less than its order whose line graph is not PMH.

\begin{proposition}\label{Prop ExpandingAllButOne}
Let $G$ be a hypohamiltonian cubic graph of odd size. Let $G'$ be a graph obtained by performing a $Y$-extension to all vertices of $G$ except one. Then, $circ(G')=|V(G')|-1$ and $L(G')$ is not PMH.
\end{proposition}
\begin{proof}
Let $v$ be the vertex of $G$ to which we do not apply a $Y$-extension, and let the resulting graph be $G'$, with the vertex of $G'$ corresponding to $v$ denoted by $v'$. Since $G$ is hypohamiltonian, $G$ admits a cycle $C$ of length $|V(G)|-1$ which passes through all the vertices of $G$ except $v$. Consequently, $G'$ admits a cycle $C'$ which passes through all the vertices of $G'$ except $v'$ and whose edges intersect the $Y$-extended triangles in the same order that $C$ passes through all the corresponding vertices in $G$, resulting in the three vertices of each $Y$-extended triangle being consecutive in $C'$. Since $G'$ is not Hamiltonian, $circ(G')=|V(G')|-1$. We proceed by supposing that $L(G')$ is PMH, for contradiction.  Denote by $Q_{v'}$ the triangle in the canonical clique partition of $L(G')$ which corresponds to the vertex $v'$. By construction of $G'$, we have $|E(G')|=|E(G)|+3(|V(G)|-1)$. Since both $|V(G)|-1$ and $|E(G)|$ are odd, $|E(G')|$ is even, i.e. $L(G')$ has even order. Moreover, since $G$ is hypohamiltonian, $G$ is bridgeless. Consequently, $G'$ is bridgeless as well, since it is obtained by applying $Y$-extensions to $G$, and so, by Lemma \ref{Lemma TechnicalPMIntersectsTriangle}, there exists a perfect matching $M$ of $L(G')$ which intersects a chosen edge of $Q_{v'}$. Lemma \ref{Lemma TechnicalUncoveredCorrespondance} assures that there exists a dominating cycle $D$ in $G'$ such that the set of its uncovered vertices does not contain $v'$.
Furthermore, the edge set of every dominating cycle of $G'$, in particular $E(D)$, intersects at least one edge of all the $Y$-extended triangles. Consequently, the dominating cycle $D$ induces a cycle in $G$ which passes through $v$ and also through every other vertex of $G$, making $G$ Hamiltonian, a contradiction.
\end{proof}

As already remarked, the graph in Figure \ref{Figure MaxDeg4NotPMH} is Hamiltonian, but not every perfect matching in its line graph can be extended to a Hamiltonian cycle. Such an example is not regular, and we are not able to find a regular one. A most natural question to ask is whether the Hamiltonicity and regularity of a graph are together sufficient conditions to guarantee the PMH-property of its line graph. Thus, we suggest the following problem.

\begin{problem}\label{hamregprob}
Let $G$ be an $r$-regular Hamiltonian graph of even size, for $r\geq 4$. Does $L(G)$ have the PMH-property?
\end{problem}

To conclude this section, let us note that not all $4$-regular (and so not all Eulerian) graphs of even size have a PMH line graph.
A non-Hamiltonian example is given in Figure \ref{Figure 4RegNotPMH}. It is not hard to check that every perfect matching of $L(G)$ which contains the edges $e_{1}e_{2}$ and $e_{3}e_{4}$ cannot be extended to a Hamiltonian cycle of $L(G)$.

\begin{figure}[ht]
      \centering
      \includegraphics[width=0.4\textwidth]{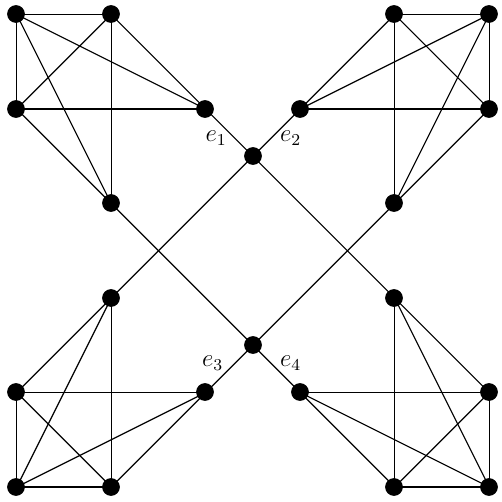}
	  \caption{A non-Hamiltonian $4$-regular graph whose line graph does not have the PMH-property.}
      \label{Figure 4RegNotPMH}
\end{figure}

Since the graphs in Figure \ref{Figure MaxDeg4NotPMH} and Figure \ref{Figure 4RegNotPMH} are both not simultaneously Eulerian and Hamiltonian, we pose a further problem.

\begin{problem}\label{hameulprob}
Let $G$ be a graph of even size which is both Eulerian and Hamiltonian. Does $L(G)$ have  the PMH-property?
\end{problem}

\section{Other classes of graphs whose line graphs are PMH}\label{otherclassesPMH}

The complete graph $K_n$, for even $n$, and the complete bipartite graph $K_{m,m}$, for $m \ge 2$, are clearly PMH. To stay in line with the contents of this paper, we now see whether their line graphs are also PMH. To this purpose, given an edge-colouring (not necessarily proper) of a Hamiltonian graph, a Hamiltonian cycle in which no two consecutive edges have the same colour will be referred to as a \emph{properly coloured Hamiltonian cycle}.

\subsection{Complete graphs}\label{Section CompleteGraphs}

First of all, we note that the line graph of a complete graph $K_n$ has a perfect matching if and only if the number of edges in $K_{n}$ is even. Hence, in the sequel we consider only complete graphs with $n\equiv 0,1 \mod 4$.

We denote the vertices of $K_{n}$ by $\{v_i: i \in [n]\}$ and the edges of $K_{n}$ by $\{e_{i,j}=v_iv_j: i\neq j \}$. Moreover, $V(L(K_{n}))$ is denoted by $\{v_{i,j}: i\neq j\}$ where the vertex $v_{i,j}$ corresponds to the edge $e_{i,j}$ of $K_n$. Finally, we denote the edges of $L(K_{n})$ by $\{e^{\,i}_{j,k}=v_{i,j}v_{i,k}:j \neq k\}$. Note that the upper index in the notation $e^{\,i}_{j,k}$ immediately indicates that the considered edge belongs to the clique $Q_i$ in the canonical clique partition of $L(K_n)$, while the order of lower indices is irrelevant.

The proof of our main theorem in this section, Theorem \ref{Theorem CompleteGraphs}, makes use of a special case of a result by Daykin \cite{Daykin1976} from 1976 which asserts the existence of a properly coloured Hamiltonian cycle if the edges of $K_{n}$ are coloured according to the following constraints.

\begin{theorem}\cite{Daykin1976}\label{Theorem Daykin1976}
If the edges of the complete graph $K_{n}$, for $n\geq 6$, are coloured in such a way that no three edges of the same colour are incident to any given vertex, then there exists a properly coloured Hamiltonian cycle.
\end{theorem}

In the following proof, the  process  of  traversing  one  path  after  another will be called \textit{concatenation of paths}. If two paths $P^1$ and $P^2$ have end-vertices $x,y$ and $y,z$, respectively, we write $P^1P^2$ to denote the path starting at $x$ and ending at $z$ obtained by traversing $P^1$ and then $P^2$.

\begin{theorem}\label{Theorem CompleteGraphs}
For $n\equiv 0,1\mod 4$, $L(K_{n})$ is PMH.
\end{theorem}

\begin{proof}
Since $K_{4}$ is Hamiltonian and cubic, by Theorem \ref{Theorem MaxDeg3}, the result holds for $n=4$. Therefore, we can assume $n>4$.

Let $M$ be a perfect matching of $L(K_{n})$. We colour the $\frac{1}{4}n(n-1)$ edges of $M$ with $\frac{1}{4}n(n-1)$ different colours. For all $e^{\,i}_{j,k} \in M$, we colour the edges $e_{i,j}$ and $e_{i,k}$ in $K_{n}$ with the same colour given to the edge $e^{\,i}_{j,k}$ in $L(K_{n})$. This gives a $P_{3}$-decomposition of $K_{n}$ in which each $P_{3}$ is monochromatic and the colours of all the 3-paths are pairwise distinct.

If $n=5$, the total number of Hamiltonian cycles in $K_{5}$ is $\frac{4!}{2}=12$. Each of the five monochromatic 3-paths in $K_{5}$ is on exactly two distinct Hamiltonian cycles. Therefore, the number of Hamiltonian cycles containing a monochromatic $P_{3}$ is at most 10, hence $K_{5}$ contains at least two (complementary) properly coloured Hamiltonian cycles.  Without loss of generality, let one of them be $H$, say $H = (v_1 , v_2, \ldots , v_5 ).$

For $n\ge 8$, by Theorem \ref{Theorem Daykin1976}, there exists a properly coloured Hamiltonian cycle $H$ in $K_{n}$ and again, without loss of generality, we can assume $H=(v_1,v_2,\ldots,v_n)$.

Now, for all $n \ge 5$ and $n\equiv 0,1\mod 4$, we will use the properly coloured Hamiltonian cycle $H$ in $K_n$ to obtain a Hamiltonian cycle $H_L$ in $L(K_n)$ containing the perfect matching $M$. We construct the Hamiltonian cycle $H_L$ in such a way that it enters and exits each clique in the canonical clique partition $\mathcal{Q}$ of $L(K_{n})$ exactly once. More precisely, we construct a suitable path $P^i$ in each clique $Q_i$ and we obtain $H_L$ as a concatenation of such paths following the order determined by $H$.
Consider the $(n-1)$-clique $Q_{i}$ and its two vertices $v_{i-1,i}$ and $v_{i,i+1}$. The corresponding edges $e_{i-1,i}$ and $e_{i,i+1}$, in $K_{n}$, are not of the same colour since they are consecutive in $H$, and so the edge $e^i_{i-1,i+1} \not\in M$. We assign a linear order $<_i$ to the set of edges $M \cap E(Q_i)$, with $(M \cap E(Q_i),<_i)=\mu_i$, such that:
\begin{itemize}
\item[(i)] if $M \cap E(Q_i)$ contains an edge incident to $v_{i-1,i}$, such an edge
is the first edge of $\mu_i$, and
\item[(ii)] if $M \cap E(Q_i)$ contains an edge incident to $v_{i,i+1}$, such an edge is the last edge of $\mu_i$,
\end{itemize}

Note that $<_{i}$ exists since $e^i_{i-1,i+1} \not\in M$. Next, we construct an $M$-alternating path in $Q_{i}$, which we denote by $P^i$, starting
at $v_{i-1,i}$ and ending at $v_{i,i+1}$ as follows: 
$P^i$ alternates between an edge of $\mu_i$ and an edge which is simultaneously adjacent to two consecutive edges in $\mu_{i}$, except possibly the first and/or last edge in $P^{i}$. Note that the choice of edges not belonging to $M\cap E(Q_{i})$ as given above is always possible since $Q_i$ is a clique. Consequently, $M \cap E(Q_i ) \subset E(P_i)$.

Now we define $H_{L}$ to be $P^{1}P^{2} \dots P^{n}$. Note that $H_L$ is a cycle since the paths $P^{i}$ are all internally and pairwise disjoint, and the beginning of $P^1$ coincides with the end of $P^n$. Moreover, $H_{L}$ is Hamiltonian because $M\subset E(H_{L})$ and so each vertex of the line graph belongs to $H_L$.
\end{proof}

\subsection{Complete bipartite graphs}\label{section kmm}
In 1976, Chen and Daykin considered an analogous version of Theorem \ref{Theorem Daykin1976} for the complete bipartite graph $K_{m,m}$ (see \cite{ChenDaykin1976}).
A particular case of Theorem $1'$ in \cite{ChenDaykin1976} can be stated as follows.

\begin{theorem}\cite{ChenDaykin1976} Consider an edge-colouring of the complete bipartite graph $K_{m,m}$ such that no vertex is incident to more than $k$ edges of the same colour. If $m\geq25k$, then there exists a properly coloured Hamiltonian cycle.
\end{theorem}

By considering the case $k=2$ in the previous theorem, i.e. $m\geq 50$, and by using an argument very similar to the one used for complete graphs in Section \ref{Section CompleteGraphs}, one could obtain that $L(K_{m,m})$ is PMH for every even $m\geq 50$. However, in a forthcoming paper, three of the authors give a more complete result and extend this by using a different and more technical approach, which goes beyond the scope of this paper. They prove the following theorem.

\begin{theorem}\label{TheoremMain}\cite{savedbytherook}
Let $m_{1}$ be an even integer and let $m_{2}\geq 1$. Then, $L(K_{m_{1},m_{2}})$ does not have the PH-property if and only if $m_{1}=2$ and $m_{2}$ is odd.
\end{theorem}

\subsection{Arbitrarily traceable graphs}
A graph $G$ is said to be \textit{arbitrarily traceable} (or equivalently \textit{randomly Eulerian}) from a vertex $v\in V(G)$ if every walk starting from $v$ and not containing any repeated edges can be completed to an Eulerian tour. This notion was firstly introduced by Ore in \cite{Ore}, who proved that an Eulerian graph $G$ is arbitrarily traceable from  $v$ if and only if every cycle in $G$ touches $v$. Here we show that every perfect matching $M$ of the line graph of an arbitrarily traceable graph can be extended to a Hamiltonian cycle.

Note that the technique used in this proof is in some way different from what was used in the case of complete graphs in Section \ref{Section CompleteGraphs}. Again, a perfect matching $M$ of $L(G)$ corresponds to a $P_3$-decomposition of $G$, but this time we construct an Euler tour of the original graph (instead of a Hamiltonian cycle) such that two edges in the same $3$-path are consecutive in the Euler tour (as opposed to what was done in Section \ref{Section CompleteGraphs} where we forbade two edges in the same $3$-path to be consecutive in the Hamiltonian cycle considered in $K_{n}$).

\begin{theorem}\label{randomeulerianPMH}
Let $G$ be a graph of even size. If $G$ is arbitrarily traceable from some vertex, then its line graph is PMH.
\end{theorem}

\begin{proof}
Let $M$ be a perfect matching of $L(G)$. Consider the $P_3$-decomposition of $G$ induced by $M$. Since $G$ is arbitrarily traceable from some vertex, there exists an Euler tour in which every pair of edges in the same $3$-path are consecutive. The sequence of edges in this Euler tour corresponds to a sequence of vertices in $L(G)$ which gives a Hamiltonian cyle $H$ of $L(G)$, and since the two edges of each 3-path in the $P_{3}$-decomposition are consecutive in the Euler tour, $H$ contains all the edges of $M$, as required.
\end{proof}

\section{Concluding remark}
Along the paper, we have proposed several sufficient conditions of different types for a graph in order to guarantee the PMH-property in its line graph. 
The wide variety of such conditions, ranging between sparse and dense graphs, do not allow us to easily identify non-trivial necessary conditions to this problem. This could be seemingly hard, but we still consider it an intriguing problem to be addressed in the future.



\end{document}